\documentclass[a4paper,11pt]{amsart}

\usepackage[margin=1in,includehead,includefoot]{geometry}
\usepackage[OT4]{fontenc}
\usepackage{amsthm,amsmath,amssymb}
\usepackage{cmtt}
\usepackage{enumitem}
\usepackage{graphicx}
\usepackage{subfigure}
\usepackage[pdftitle={Triangle-free geometric intersection graphs with large chromatic number},
            pdfauthor={J. Kozik, T. Krawczyk, M. Lason, P. Micek, A. Pawlik, W. T. Trotter, B. Walczak},
            colorlinks=true,linkcolor=black,citecolor=black,filecolor=black,urlcolor=black]{hyperref}

\newtheorem{theorem}{Theorem}
\newtheorem{lemma}[theorem]{Lemma}
\newtheorem{fact}[theorem]{Fact}
\theoremstyle{definition}
\newtheorem{problem}{Problem}

\newcommand{\cgF}{\mathcal{F}}
\newcommand{\cgI}{\mathcal{I}}
\newcommand{\cgP}{\mathcal{P}}

\let\leq\leqslant
\let\geq\geqslant
\let\setminus\smallsetminus
\let\epsilon\varepsilon

\def\NN{\mathbb{N}}
\def\RR{\mathbb{R}}

\DeclareMathOperator{\interior}{int}

\renewenvironment{enumerate}{\begin{enumorig}[label=\textup{(\roman*)}, noitemsep, topsep=1.5mm plus 1.5mm, leftmargin=*, widest=iii]}{\end{enumorig}}

\makeatletter
\let\old@setaddresses\@setaddresses
\def\@setaddresses{\bigskip\bgroup\parindent 0pt\let\scshape\relax\old@setaddresses\egroup}
\makeatother

\title{Triangle-free geometric intersection graphs with~large~chromatic number}

\author[A. Pawlik\and J. Kozik\and T. Krawczyk\and M. Laso\'n\and P. Micek\and W. T. Trotter\and B. Walczak]{Arkadiusz Pawlik\and Jakub Kozik\and Tomasz Krawczyk\and Micha\l{} Laso\'n\and Piotr Micek\and William T. Trotter\and Bartosz Walczak}

\thanks{A journal version of this paper appeared in \emph{Discrete Comput.\ Geom.}, 50(3):714--726, 2013.}
\thanks{Five authors were supported by Ministry of Science and Higher Education of Poland grant 884/N-ESF-EuroGIGA/10/2011/0 within ESF EuroGIGA project GraDR\@.}

\address[Arkadiusz Pawlik, Jakub Kozik, Tomasz Krawczyk, Piotr Micek]{Theoretical Computer Science Department, Faculty of Mathematics and Computer Science, Jagiellonian University, Krak\'ow, Poland}
\email{\mtt\{pawlik,jkozik,krawczyk,micek\mtt\}@tcs.uj.edu.pl}

\address[Micha\l{} Laso\'n]{Theoretical Computer Science Department, Faculty of Mathematics and Computer Science, Jagiellonian University, Krak\'ow, Poland; Institute of Mathematics of the Polish Academy of Sciences, Warsaw, Poland}
\email{michalason@gmail.com}

\address[William T. Trotter]{School of Mathematics, Georgia Institute of Technology, Atlanta, GA 30332, USA}
\email{trotter@math.gatech.edu}

\address[Bartosz Walczak]{Theoretical Computer Science Department, Faculty of Mathematics and Computer Science, Jagiellonian University, Krak\'ow, Poland; \'Ecole Polytechnique F\'ed\'erale de Lausanne, Switzerland}
\email{walczak@tcs.uj.edu.pl}

\begin{document}

\begin{abstract}
Several classical constructions illustrate the fact that the chromatic number of a graph can be arbitrarily large compared to its clique number.
However, until very recently, no such construction was known for intersection graphs of geometric objects in the plane.
We provide a general construction that for any arc-connected compact set $X$ in $\RR^2$ that is not an axis-aligned rectangle and for any positive integer $k$ produces a family $\cgF$ of sets, each obtained by an independent horizontal and vertical scaling and translation of $X$, such that no three sets in $\cgF$ pairwise intersect and $\chi(\cgF)>k$.
This provides a negative answer to a question of Gy\'arf\'as and Lehel for L-shapes.
With extra conditions, we also show how to construct a triangle-free family of homothetic (uniformly scaled) copies of a set with arbitrarily large chromatic number.
This applies to many common shapes, like circles, square boundaries, and equilateral L-shapes.
Additionally, we reveal a surprising connection between coloring geometric objects in the plane and on-line coloring of intervals on the line.
\end{abstract}

\maketitle

\section{Introduction}

A \emph{proper coloring} of a graph is an assignment of colors to the vertices of the graph such that no two adjacent ones are assigned the same color.
The minimum number of colors sufficient to color a graph $G$ properly is called the \emph{chromatic number} of $G$ and denoted by $\chi(G)$.
The \emph{clique number} of $G$, denoted by $\omega(G)$, is the maximum size of a set of pairwise adjacent vertices (a clique) in $G$.
A graph is \emph{triangle-free} if it contains no clique of size $3$.

It is clear that $\chi(G)\geq\omega(G)$.
On the one hand, the chromatic and clique numbers of a graph can be arbitrarily far apart.
There are various constructions of graphs that are triangle-free and still have arbitrarily large chromatic number.
The first one was given in 1949 by Zykov \cite{Zyk49}, and the one perhaps best known is due to Mycielski \cite{Myc55}.
On the other hand, in many important classes of graphs the chromatic number is bounded in terms of the clique number.
Graphs $G$ for which this bound is as tight as possible, that is, $\chi(H)=\omega(H)$ holds for every induced subgraph $H$ of $G$, are called \emph{perfect}.
They include bipartite graphs, split graphs, chordal graphs, interval graphs, comparability graphs, etc.
A class of graphs is \emph{$\chi$-bounded} if there is a function $f\colon\NN\to\NN$ such that $\chi(G)\leq f(\omega(G))$ holds for any graph $G$ from the class\footnotemark.
\footnotetext{This notion has been introduced by Gy\'arf\'as \cite{Gya87}, who called the class \emph{$\chi$-bound} and the function \emph{$\chi$-binding}.
However, the term \emph{$\chi$-bounded} seems to be better established in the modern terminology.}
In particular, the class of perfect graphs is $\chi$-bounded by the identity function.

In this paper, we focus on the relation between the chromatic number and the clique number for classes of graphs arising from geometry.
The \emph{intersection graph} of a family of sets $\cgF$ is the graph with vertex set $\cgF$ and edge set consisting of pairs of intersecting elements of $\cgF$.
We consider families $\cgF$ consisting of arc-connected compact subsets of $\RR^d$ (mostly $\RR^2$) of a common kind (segments, curves, polygons, etc.).
For simplicity, we identify the family $\cgF$ with its intersection graph.

In the simple one-dimensional case of subsets of $\RR$, the only connected sets are intervals.
They define the class of \emph{interval graphs}, which are well known to be perfect.
The situation becomes much more complex in higher dimensions.

The study of the chromatic number of intersection graphs of geometric objects in $\RR^2$ was initiated in the seminal paper of Asplund and Gr\"unbaum \cite{AG60}, where they proved that the families of axis-aligned rectangles are $\chi$-bounded.
Specifically, they proved that every family $\cgF$ of axis-aligned rectangles in the plane satisfies $\chi(\cgF)\leq 4\omega(\cgF)^2-3\omega(\cgF)$.
This was later improved by Hendler \cite{Hen98} to $\chi(\cgF)\leq 3\omega(\cgF)^2-2\omega(\cgF)-1$.
No construction of families of rectangles with $\chi$ superlinear in terms of $\omega$ is known.
On the other hand, Burling \cite{Bur65} showed that triangle-free intersection graphs of axis-aligned boxes in $\RR^3$ can have arbitrarily large chromatic number.

Gy\'arf\'as \cite{Gya85,Gya86} proved $\chi$-boundedness of \emph{overlap graphs}, that is, graphs admitting a representation by closed intervals on the line such that the edges correspond to the pairs of intervals that intersect but are not nested.
Alternatively, the same class of graphs can be defined as the intersection graphs of families of chords of a circle.
Gy\'arf\'as's proof yields the bound $\chi(\cgF)\leq 2^{\omega(\cgF)}(2^{\omega(\cgF)}-2)\omega(\cgF)^2$ for any such family $\cgF$.
This was improved and generalized by Kostochka and Kratochv\'{\i}l \cite{KK97}, who showed that every family $\cgF$ of convex polygons inscribed in a circle satisfies $\chi(\cgF)<2^{\omega(\cgF)+6}$.

Paul Erd\H{o}s asked in the 1970s\footnote{An approximate date confirmed in personal communication with Andr\'as Gy\'arf\'as and J\'anos Pach; see also \cite[Problem 1.9]{Gya87} and \cite[Problem 2 in Section 9.6]{BMP-book}.} whether the class of intersection graphs of line segments in the plane is $\chi$-bounded, or more specifically, whether the chromatic number of triangle-free intersection graphs of line segments is bounded by an absolute constant.
In a previous paper \cite{PKK+14}, we show that the answer is negative.
Namely, for every positive integer $k$, we construct a family $\cgF$ of line segments in the plane with no three pairwise intersecting segments and such that $\chi(\cgF)>k$.

In this paper, we generalize that construction to a wide class of families of sets in the Euclidean plane.
Let $X$ be an arc-connected compact set in $\RR^2$ that is not an axis-aligned rectangle.
For every positive integer $k$, we present a triangle-free family $\cgF$ of sets, each obtained by translation and independent horizontal and vertical scaling of $X$, such that $\chi(\cgF)>k$.

This applies to a wide range of geometric shapes like axis-aligned ellipses, rhombuses, rectangular frames, cross-shapes, L-shapes, etc.
For some shapes (e.g.\ circles and square boundaries), it is even possible to restrict the allowed transformations to uniform scaling and translation.
This contrasts with the result of Kim, Kostochka and Nakprasit \cite{KKN04} that every family $\cgF$ of homothetic (uniformly scaled) copies of a fixed convex compact set in the plane satisfies $\chi(\cgF)\leq 6\omega(\cgF)-6$.

Our result also gives a negative answer to the question of Gy\'arf\'as and Lehel \cite{GL85} whether the families of axis-aligned L-shapes, that is, shapes consisting of a horizontal and a vertical segments of arbitrary lengths forming the letter `L', are $\chi$-bounded.
We prove that this is not the case even for equilateral L-shapes.
However, we are unable to classify all arc-connected compact sets whose families of homothets are $\chi$-bounded.
This is discussed in Section \ref{sec:homothetic}.

There is some evidence that unrestricted scaling is the key property necessary to make the chromatic number large while keeping the clique number small.
For instance, Suk \cite{Suk14} proved that families of unit-length segments are $\chi$-bounded.
This result easily generalizes to the case that the ratio between the maximum and minimum lengths of segments in the family is bounded.

In Section \ref{sec:online}, we discuss how bounding the chromatic number of families of geometric objects in the plane is related to a specific on-line coloring problem for overlap graphs.
In particular, we give an alternative presentation of the construction of triangle-free families of \emph{rectangular frames} (that is, boundaries of axis-aligned rectangles) with arbitrarily large chromatic number.

\section{Translation and independent horizontal and vertical scaling}\label{sec:independent-scaling}

\begin{theorem}\label{thm:independent-scaling}
For every arc-connected compact set\/ $X\subset\RR^2$ that is not an axis-aligned rectangle and every integer\/ $k\geq 1$, there is a triangle-free family\/ $\cgF$ of sets in the plane, each obtained by translation and independent horizontal and vertical scaling of\/ $X$, such that\/ $\chi(\cgF)>k$.
\end{theorem}

Define the \emph{bounding box} of a non-empty bounded set (family of sets) in the plane to be the smallest axis-aligned rectangle that contains it (all the sets in the family).
We say that a curve \emph{stabs} an axis-aligned rectangle \emph{horizontally} (\emph{vertically}) if it is contained in that rectangle and connects its left and right (top and bottom, respectively) boundaries.
We need the following simple properties of stabbing curves.

\begin{fact}\label{fact:curve-cut}
Let\/ $R$ and\/ $S$ be two axis-aligned rectangles such that\/ $R\subset S$ and\/ $R$ touches the left and right (top and bottom) sides of\/ $S$.
If a curve\/ $c$ stabs\/ $S$ vertically (horizontally), then there is a curve\/ $c'\subset c$ that stabs\/ $R$ vertically (horizontally).
\end{fact}

\begin{fact}\label{fact:curve-cross}
Let\/ $R$ be an axis-aligned rectangle, $c_V$ be a curve that stabs\/ $R$ vertically, and\/ $c_H$ be a curve that stabs\/ $R$ horizontally.
Then\/ $c_V\cap c_H\neq\emptyset$.
\end{fact}

\begin{figure}[t]
\begin{center}
\includegraphics[scale=.55]{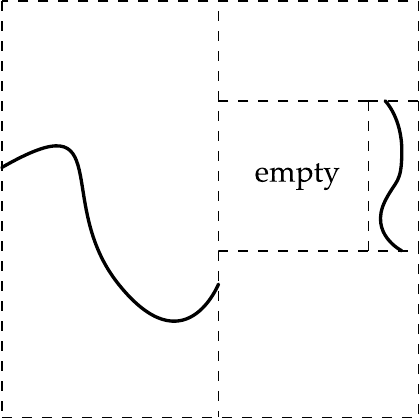}
\end{center}
\caption{Empty rectangle, left and right stabbers}
\label{fig:property}
\end{figure}

To prove Theorem \ref{thm:independent-scaling}, we use the following properties of the set $X$ or its horizontal reflection (see Figure~\ref{fig:property}):
\begin{enumerate}
\item\label{is-cond:i} $X$ is bounded.
Let $U$ be the bounding box of $X$.
\item\label{is-cond:ii} There is an axis-aligned rectangle $E$ contained in the interior of $U$ and disjoint from $X$.
We fix any such $E$ and call it the \emph{empty rectangle} of $X$.
\item\label{is-cond:iii}
Let $V_L$ be the rectangle, situated to the left of $E$, defined by the top, bottom, and left edges of $U$ and the line through the left edge of $E$.
There is a curve in $X$ that stabs $V_L$ horizontally.
We fix any such curve and call it the \emph{left stabber} of $X$.
\item\label{is-cond:iv}
Let $V_R$ be the rectangle, situated to the right of $E$, defined by lines through the top, bottom, and right edges of $E$ and the right edge of $U$.
There is a curve in $X$ that stabs $V_R$ vertically.
We fix any such curve and call it the \emph{right stabber} of $X$.
\end{enumerate}
The condition \ref{is-cond:i} follows from the assumption that $X$ is compact; \ref{is-cond:ii} is a consequence of the fact that $(\interior U)\setminus X$ is non-empty and open; \ref{is-cond:iii} can be proved quickly as follows.
There must be a curve $c\subset X$ that stabs $U$ horizontally, because $U$ is compact and arc-connected.
Thus there is a curve that stabs $V_L$ horizontally, by Fact \ref{fact:curve-cut}.
The condition \ref{is-cond:iv} must be true either for $X$ or the horizontal reflection of $X$.
This can be proved as follows.
There must be a curve $c\subset X$ that stabs $U$ vertically.
By Fact \ref{fact:curve-cut}, there is a curve $c'\subset c$ that stabs vertically the rectangle defined by the horizontal sides of $E$ and vertical sides of $U$.
Because $c'$ does not touch the empty rectangle and is arc-connected, it is a valid right stabber of $X$ or its horizontal reflection is a valid right stabber of the horizontal reflection of $X$.
We can assume without loss of generality that condition \ref{is-cond:iv} is true for $X$, as otherwise the horizontal reflection of $X$ would satisfy conditions \ref{is-cond:i}--\ref{is-cond:iv}.

For convenience, we use the term \emph{transformed copy} of $X$ to denote any set obtained from $X$ by independent horizontal and vertical scaling and translation.
Note that we will use only positive scale factors (that is, we will not need reflected copies of $X$).

Let $\cgF$ be a family of transformed copies of $X$ and let $R$ be the bounding box of $\cgF$.
An axis-aligned rectangle $P$ is a \emph{probe} for $\cgF$ if the following conditions are satisfied:
\begin{enumerate}
\item $P$ is contained in $R$ and touches the right side of $R$.
\item The sets in $\cgF$ that intersect $P$ are pairwise disjoint.
\item\label{probe-cond:iii} Every set in $\cgF$ that intersects $P$ contains a curve that stabs $P$ vertically.
\item There is a vertical line that cuts $P$ into two rectangles the left of which is disjoint from all the sets in $\cgF$.
For each probe, we fix any such vertical line and call the left of the two rectangles the \emph{root} of $P$.
\end{enumerate}

We define sequences $(s_i)_{i\in\NN}$ and $(p_i)_{i\in\NN}$ by induction, setting $s_1=p_1=1$, $s_{i+1}=(p_i+1)\cdot s_i+p_i^2$, and $p_{i+1}=2p_i^2$.
Now we are ready to state and prove the following lemma, which immediately implies Theorem \ref{thm:independent-scaling}.

\begin{lemma}\label{lem:other-induction}
For every integer\/ $k\geq 1$, there is a triangle-free family\/ $\cgF(k)$ of\/ $s_k$ transformed copies of\/ $X$ and a family\/ $\cgP(k)$ of\/ $p_k$ pairwise disjoint probes for\/ $\cgF(k)$ such that every proper coloring of\/ $\cgF(k)$ uses at least\/ $k$ colors on the members of\/ $\cgF(k)$ intersecting some probe in\/~$\cgP(k)$.
\end{lemma}

\begin{proof}
Set $\cgF(1)=\{X\}$.
Let $P$ be the probe for $\cgF(1)$ obtained by horizontal extension of the empty rectangle $E$ of $X$ to the right side of the bounding box of $X$.
Note that $P$ is stabbed vertically by the right stabber of $X$ and $E$ is a valid root of $P$.
Set $\cgP(1)=\{P\}$.
It is clear that the statement of the lemma holds.

Now, let $k\geq 2$, and assume by induction that $\cgF(k-1)$ and $\cgP(k-1)$ exist and satisfy the conditions of the lemma.
We show how to construct $\cgF(k)$ and $\cgP(k)$.
For each probe $P\in\cgP(k-1)$, split it vertically into two parts $P^\uparrow$ and $P^\downarrow$, with a little vertical margin in-between to make them disjoint.

First, we construct a helper family $\cgF'$ containing $\cgF(k-1)$ and, for each $P\in\cgP(k-1)$, an additional transformed copy $D_P$ of $X$ called the \emph{diagonal}.
The diagonal $D_P$ is placed as follows.
Let $w_1$ be the distance between the left side of the empty rectangle of $X$ and the left side of the bounding box of $X$, and let $w_2$ be the width of the bounding box of $X$.
Take a copy of $X$ transformed in such a way that its bounding rectangle is equal to $P^\uparrow$, and then scale it horizontally by $\frac{2w_2}{w_1}$ keeping the left side of the bounding rectangle fixed.
The resulting set is the diagonal $D_P$.
The chosen scale factor guarantees that the empty rectangle of $D_P$ lies completely to the right of the bounding box of $\cgF(k-1)$.
For an illustration, see Figure \ref{fig:diagonal}.

\begin{figure}[t]
\begin{center}
\includegraphics[scale=.3]{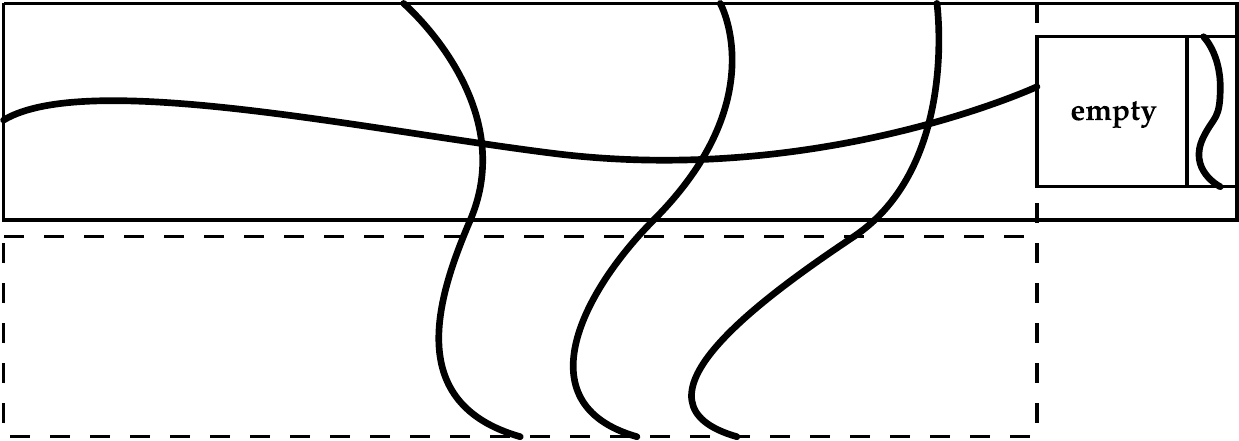}
\end{center}
\caption{The diagonal of a probe $P$ (top) and the rectangle $P^\downarrow$ (bottom) pierced by vertical stabbers of $P$}
\label{fig:diagonal}
\end{figure}

It is easy to see that $D_P$ intersects exactly the same sets in $\cgF(k-1)$ as $P^\uparrow$ does.
Every such set actually contains a curve that stabs $P$ (and thus $P^\uparrow$, by Fact \ref{fact:curve-cut}) vertically, by the probe condition \ref{probe-cond:iii}.
Furthermore, we know that the left stabber of $D_P$ stabs $P^\uparrow$ horizontally, thanks to the chosen scale factor.
By Fact \ref{fact:curve-cross}, the left stabber of $D_P$ intersects every set in $\cgF(k-1)$ that intersects $P^\uparrow$.

The size of $\cgF'$ is $s_{k-1}+p_{k-1}$.

We now construct the family $\cgF(k)$ in two steps.
First, draw a copy $\cgF_{\text{outer}}$ of $\cgF(k-1)$, called the \emph{outer} family, associated with a set of probes $\cgP_{\text{outer}}$, a copy of $\cgP(k-1)$.
Then, for each $P\in\cgP_{\text{outer}}$, place a scaled copy $\cgF_P'$ of $\cgF'$, called an \emph{inner} family, inside the root of $P$.
Note that $\cgF(k)$ is of size $s_{k-1} + p_{k-1}(s_{k-1}+p_{k-1}) = s_k$.

Now we show how the set of probes $\cgP(k)$ is constructed.
For each probe $P\in\cgP_{\text{outer}}$, let $\cgF_P$ be the family obtained from $\cgF_P'$ by removing the diagonals.
Note that $\cgF_P$ is a transformed copy of $\cgF(k-1)$, so let $\cgP_P$ be the associated set of (transformed) probes.
For each probe $P\in\cgP_{\text{outer}}$ and each probe $Q\in\cgP_P$, introduce two probes to $\cgP(k)$:\ the \emph{upper} probe $U_{P,Q}$ and the \emph{lower} probe $L_{P,Q}$.
The upper probe $U_{P,Q}$ is the horizontal extension of the empty rectangle of the diagonal $D_Q$ to the right side of the bounding box of $\cgF(k)$.
The lower probe $L_{P,Q}$ is $Q^\downarrow$ extended horizontally to the right side of the bounding box of $\cgF(k)$.

To see that $U_{P,Q}$ is a probe for $\cgF(k)$, note that $U_{P,Q}$ intersects $D_Q$, all the sets in $\cgF_{\text{outer}}$ pierced by $P$, and nothing else.
First, the empty rectangle of $D_P$ is completely to the right of all sets in $\cgF_P$, so the only set in $\cgF_P'$ intersecting the upper probe is $D_P$ (and the right stabber of $D_P$ stabs $U_{P,Q}$ vertically).
Second, since $U_{P,Q}$ is contained in $P$ and starts in its root, by Fact \ref{fact:curve-cut}, all the sets in $\cgF_{\text{outer}}$ pierced by $P$ contain curves that stab $U_{P,Q}$ vertically.
The inner family was built inside the root of a probe $P$ of the outer family, thus all the sets piercing $U_{P,Q}$ are disjoint.
The empty rectangle of $D_P$ is a valid root of $U_{P,Q}$.

The proof that $L_{P,Q}$ is a probe is analogous.
$L_{P,Q}$ intersects all the sets in $\cgF_{\text{outer}}$ pierced by $P$, all the sets in $\cgF_P$ pierced by $Q$, and nothing else.
In particular, $L_{P,Q}$ is disjoint from $D_Q$.
Since the sets in the inner family are disjoint from the sets in the outer family, all the sets intersecting $L_{P,Q}$ are disjoint.
Furthermore, they contain curves that stab $L_{P,Q}$ vertically as a consequence of Fact \ref{fact:curve-cut}.
The intersection of the root of $Q$ with $L_{P,Q}$ is a valid root of $L_{P,Q}$.

The size of $\cgP(k)$ is $2p_{k-1}^2=p_k$.

The family $\cgF(k)$ is triangle-free, because we constructed it by taking disjoint copies of triangle-free families and adding disjoint diagonals intersecting independent sets.
Let $\phi$ be a proper coloring of $\cgF(k)$.
We show that there is a probe in $\cgP(k)$ for which $\phi$ uses at least $k$ colors on the sets in $\cgF(k)$ intersecting that probe.
Consider the restriction of $\phi$ to the outer family $\cgF_{\text{outer}}$.
There is a probe $P\in\cgP_{\text{outer}}$ such that $\phi$ uses at least $k-1$ colors on the sets in $\cgF_{\text{outer}}$ intersecting $P$.
Now, consider $\cgF_P$, the inner copy of $\cgF(k-1)$ put inside the root of $P$.
Again, there is a probe $Q\in\cgP_P$ such that $\phi$ uses at least $k-1$ colors on the sets in $\cgF_P$ intersecting $Q$.
If $\phi$ uses different sets of colors on the sets in $\cgF_{\text{outer}}$ intersecting $P$ and the sets in $\cgF_P$ intersecting $Q$, then at least $k$ colors are used on the sets pierced by the lower probe $L_{P,Q}$.
Otherwise, if $\phi$ uses the same set of colors on these two families, then another color must be used on the diagonal $D_Q$, and thus $\phi$ uses at least $k$ colors on the sets intersecting the upper probe $U_{P,Q}$.
\end{proof}

The lemma provides only the weak inequality $\chi(\cgF)\geq k$ instead of the strong one required by Theorem \ref{thm:independent-scaling}.
This is on purpose---we can add diagonals to the family $\cgF(k)$ from the statement of the lemma (the same way as $\cgF'$ is constructed in the proof of the lemma) and increase the chromatic number by one.
This way we obtain the smallest family $\cgF$ satisfying Theorem \ref{thm:independent-scaling} that we know of.
Its size is $s_k+p_k$.
The inductive definition of $p_k$ and $s_k$ yields
\begin{equation*}
2^{2^{k-1}-1}=p_k\leq s_k\leq 2^{2^{k-1}}-1,
\end{equation*}
and thus $s_k+p_k=\Theta(2^{2^{k-1}})$.

\section{Uniform scaling and translation}\label{sec:homothetic}

We adapt our construction from the previous section to work with uniform scaling for a base shape that meets additional conditions.
The reader is advised to read the previous section first, as we do not repeat all the details here.

We start with technical conditions capturing the full generality of the sets that admit our construction for uniform scaling.
Let $X$ be a subset of $\RR^2$.
We say that $X$ is \emph{anchored} if it can be affinely transformed to a set $X'$ with the following properties:
\begin{enumerate}
\item\label{us-cond:i}
$X'$ is contained in the square $U=[0,1]\times(0,1)$, called the \emph{bounding square} of $X'$.
Note that $U$ is open at the top and bottom.
\item\label{us-cond:ii}
For every $\epsilon\in(0,1)$, there is a closed square $E(\epsilon)\subset U$ such that $E(\epsilon)$ is disjoint from $X'$, the width $\xi(\epsilon)$ of $E(\epsilon)$ satisfies $(1+\epsilon)\xi(\epsilon)<\epsilon$, and the distance between the right side of $E(\epsilon)$ and the right side of $U$ is equal to $\epsilon\xi(\epsilon)$.
We fix $E(\epsilon)$ and call it the \emph{$\epsilon$-empty square} of $X'$.
\item\label{us-cond:iii}
Let $V_L(\epsilon)$ be the rectangle formed by the top, bottom, and left sides of $U$ and the line through the left side of $E(\epsilon)$.
There is a curve in $X'$ that stabs $V_L(\epsilon)$ horizontally.
We fix any such curve and call it the \emph{left\/ $\epsilon$-stabber} of $X'$.
\item\label{us-cond:iv}
Let $V_R(\epsilon)$ be the rectangle formed by lines through the top, bottom, and right sides of $E(\epsilon)$ and the right side of $U$.
There is a curve in $X'$ that stabs $V_R(\epsilon)$ vertically.
We fix any such curve and call it the \emph{right\/ $\epsilon$-stabber} of $X'$.
\end{enumerate}
Note that the intersection graphs of homothetic copies of $X'$ are exactly the intersection graphs of homothetic copies of $X$.
Therefore, for simplicity, we assume that $X$ itself satisfies \ref{us-cond:i}--\ref{us-cond:iv}.

Examples of anchored sets include squares, circles, half-circles, boundaries of convex sets, and L-shapes.
Figures \ref{fig:anchorex3} and \ref{fig:anchorex4} show two sets that are not anchored and in fact are $\chi$-bounded\footnote{We leave this statement without a proof, as we are unable to prove it for all non-anchored sets.}.
We do not know of any non-anchored arc-connected compact set the families of homothetic copies of which are not $\chi$-bounded.

\begin{figure}[t]
\begin{center}
\subfigure[\unskip]{\label{fig:anchorex1}\includegraphics[scale=.45]{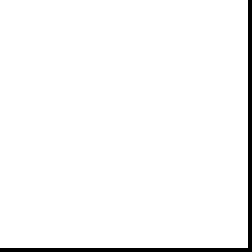}}
\hspace{0.5cm}
\subfigure[\unskip]{\label{fig:anchorex2}\includegraphics[scale=.45]{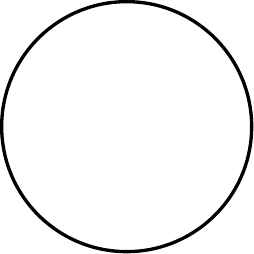}}
\hspace{0.5cm}
\subfigure[\unskip]{\label{fig:anchorex3}\includegraphics[scale=.45]{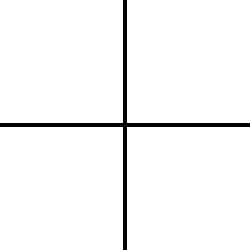}}
\hspace{0.5cm}
\subfigure[\unskip]{\label{fig:anchorex4}\includegraphics[scale=.45]{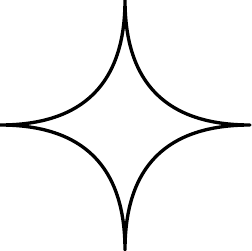}}
\end{center}
\caption{\subref{fig:anchorex1}, \subref{fig:anchorex2} anchored sets; \subref{fig:anchorex3}, \subref{fig:anchorex4} non-anchored sets}
\label{fig:anchorex}
\end{figure}

\begin{theorem}\label{thm:uniform}
For every anchored subset\/ $X$ of\/ $\RR^2$ and every integer\/ $k\geq 1$, there is a triangle-free family\/ $\cgF$ of homothetic copies of\/ $X$ in the plane such that\/ $\chi(\cgF)>k$.
\end{theorem}

Let $\cgF$ be a family of homothetic copies of $X$ and $B$ be the bounding box of $\cgF$.
An axis-aligned rectangle $P$ is an \emph{$\epsilon$-probe} for $\cgF$ if the following conditions are satisfied:
\begin{enumerate}
\item $P$ is contained in $B$ and touches the right side of $B$.
\item The elements of $\cgF$ that intersect $P$ are pairwise disjoint.
\item Every set in $\cgF$ that intersects $P$ contains a curve that stabs $P$ vertically.
\item There is a vertical line cutting $P$ into two rectangles the left of which, called the \emph{root} of $P$, is a square disjoint from all the members of $\cgF$.
\item The ratio between the width and height of $P$ is equal to $1+\epsilon$.
\end{enumerate}

\begin{lemma}\label{lem:anchor-induction}
For every integer\/ $k\geq 1$ and every\/ $\epsilon\in(0,1)$, there is a triangle-free family\/ $\cgF(k,\epsilon)$ of\/ $s_k$ homothetic copies of\/ $X$ and a family\/ $\cgP(k,\epsilon)$ of\/ $p_k$ pairwise disjoint\/ $\epsilon$-probes for\/ $\cgF(k,\epsilon)$ such that every proper coloring of\/ $\cgF(k,\epsilon)$ uses at least\/ $k$ colors on the members of\/ $\cgF(k,\epsilon)$ intersecting some probe in\/ $\cgP(k,\epsilon)$.
\end{lemma}

\begin{proof}
For every $\epsilon\in(0,1)$, set $\cgF(1,\epsilon)=\{X\}$.
Extend the $\epsilon$-empty square of $X$ to the right side of the bounding box of $X$ to form a probe for $\cgF(1,\epsilon)$, and let this probe be the only member of $\cgP(1,\epsilon)$.
The families $\cgF(1,\epsilon)$ and $\cgP(1,\epsilon)$ clearly satisfy the conditions of the lemma.

Now, let $k\geq 2$, and assume by induction that for every $\epsilon\in(0,1)$, the families $\cgF(k-1,\epsilon)$ and $\cgP(k-1,\epsilon)$ exist and satisfy the conditions of the lemma.
Fix any $\epsilon\in(0,1)$.
Let $\cgF_0=\cgF(k-1,\frac{\epsilon}{8})$ and $\cgP_0=\cgP(k-1,\frac{\epsilon}{8})$.
Let $s_{\min}$ and $s_{\max}$ be respectively the minimum and maximum width of the root of a probe in $\cgP_0$.
Let $m=\frac{\epsilon}{8}\cdot s_{\min}$ and $\epsilon_1=\frac{2m}{s_{\max}}$.
Note that $\epsilon_1\leq\frac{\epsilon}{2}$.
For each probe $P\in\cgP_0$, define a homothetic copy $D_P$ of $X$, called the \emph{diagonal} of $P$, as follows.
Let $s$ be the width of the root of $P$.
Thus $s_{\min}\leq s\leq s_{\max}$.
Split the root of $P$ into four equal-size quadrants (see Figure \ref{fig:induction}).
Within the top right quadrant, put a homothetic copy of $X$ so that its bounding square fills the quadrant, and then shift it to the right by $\frac{\epsilon}{8}\cdot s+m$.
The resulting set is $D_P$.
This way, each diagonal sticks out of the bounding box of $\cgF_0$ by exactly $m$.
The $\epsilon_1$-empty square of $D_P$ has width $\xi(\epsilon_1)\cdot\frac{s}{2}$.
Since $(1+\epsilon_1)\xi(\epsilon_1)\cdot\frac{s}{2}<\epsilon_1\cdot\frac{s}{2}\leq m$, the $\epsilon_1$-empty square of $D_P$ lies entirely to the right of the bounding box of $\cgF_0$.
Since $\frac{\epsilon}{8}\cdot s+m\leq\tfrac{\epsilon}{2}\cdot\frac{s}{2}$, the root of $P$ and the diagonal $D_P$ have non-empty intersection.
Consequently, the left $\epsilon_1$-stabber of $D_P$ (and thus $D_P$ itself) intersects all the members of $\cgF_0$ intersecting $P$.
Define $\cgF'$ to be the family $\cgF_0$ together with the diagonals $D_P$ for all probes $P\in\cgP_0$.

\begin{figure}[t]
\begin{center}
\includegraphics[scale=.55]{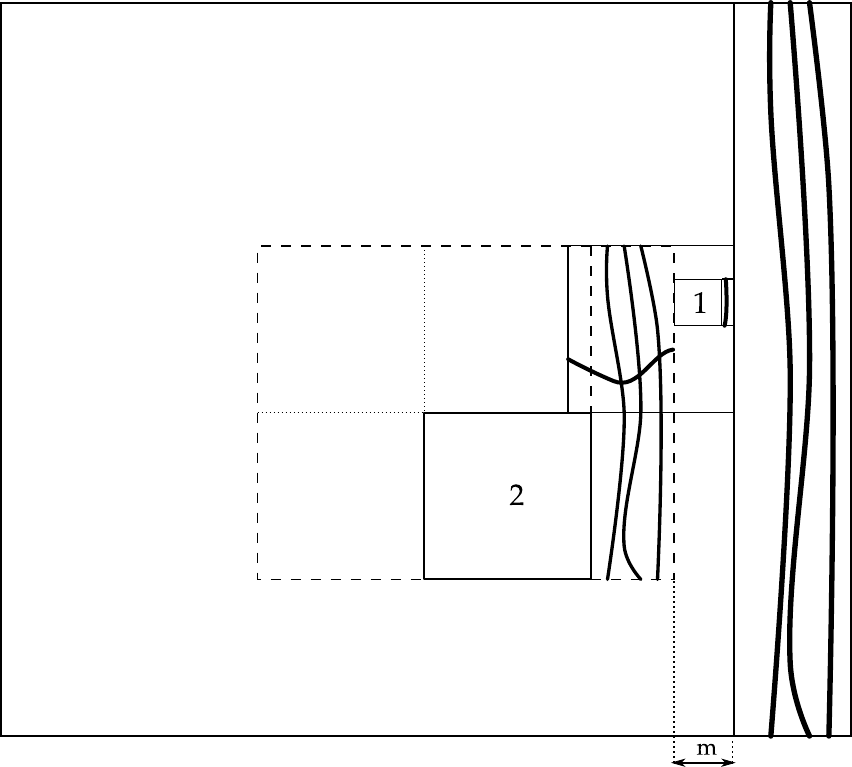}
\end{center}
\caption{The diagonal of an $\epsilon$-probe $P$ placed inside the root of an outer probe.
The squares marked 1 and 2 are the upper and lower root.}
\label{fig:induction}
\end{figure}

Now, for each probe $P\in\cgP_0$, define two squares:\ the \emph{upper root} and the \emph{lower root}.
The upper root of $P$ is the $\epsilon_1$-empty square of $D_P$.
The lower root of $P$ is the lower right quadrant of the root of $P$.
The upper and lower roots of $P$ are disjoint from all the sets in $\cgF'$.
The distance of the lower root to the right side of the bounding box of $\cgF'$ is $\frac{\epsilon}{8}\cdot s+m$, which is at most $\frac{\epsilon}{2}\cdot\frac{s}{2}$.
Therefore, the distance of each upper or lower root to the right side of the bounding box of $\cgF'$ is at most the $\frac{\epsilon}{2}$ fraction of its width.

We construct $\cgF(k,\epsilon)$ as follows.
Let $s$ be the minimum width of a square containing $\cgF'$, and let $t$ be the minimum size of a lower or upper root.
Let $\cgF_{\text{outer}}=\cgF(k,\frac{\epsilon t}{2s})$ and $\cgP_{\text{outer}}=\cgP(k,\frac{\epsilon t}{2s})$.
Then, for each $P\in\cgP_{\text{outer}}$, with root of width $r$, define $\cgF_P$ to be a homothetic copy of $\cgF'$ scaled by $\frac{r}{s}$ and placed inside the root of $P$ so that the bounding box of $\cgF_P$ touches the right side of the root.
Define $\cgF(k,\epsilon)$ to be the union of $\cgF_{\text{outer}}$ and $\cgF_P$ for all $P\in\cgP_{\text{outer}}$.

Now, we show how $\cgP(k,\epsilon)$ is constructed.
Consider the upper or lower root $R$ of any $\cgF_P$ with $P\in\cgP_{\text{outer}}$.
The ratio between the distance of $R$ to the right side of the bounding box of $\cgF(k,\epsilon)$ and the width of $R$ is at most $\frac{\epsilon}{2}+\frac{\epsilon t}{2s}\cdot\frac{s}{t}=\epsilon$.
Therefore, there is an $\epsilon$-probe for $\cgF(k,\epsilon)$ whose root is contained in $R$.
Choose one such probe for each upper or lower root $R$ of each $\cgF_P$ to form $\cgP(k,\epsilon)$.

That $\cgF(k,\epsilon)$ and $\cgP(k,\epsilon)$ satisfy the requirements of the lemma is proved the same way as in the proof of Lemma \ref{lem:other-induction}.
\end{proof}

\section{On-line coloring of overlap graphs}\label{sec:online}

Recall that an overlap graph is a graph whose vertices are intervals in $\RR$ and edges are the pairs of intervals that intersect but are not nested.
To provide a different insight into our construction from the previous sections, we show how it arises from a strategy for an on-line coloring game played on triangle-free overlap graphs.

The game is played by two players:\ Presenter and Painter.
Presenter presents intervals one at a time, according to the following two restrictions:
\begin{enumerate}
\item the intervals are presented in the increasing order of their left endpoints;
\item the overlap graph represented by all presented intervals is triangle-free.
\end{enumerate}
Right after an interval is presented, Painter has to label it with a color distinct from the colors of its neighbors in the presented overlap graph.
In other words, Painter produces a proper coloring of the graph.
The color once assigned to a vertex cannot be changed later on.

\begin{theorem}\label{thm:online}
For every integer\/ $k\geq 1$, Presenter has a strategy that forces Painter to use more than\/ $k$ colors in the on-line coloring game on triangle-free overlap graphs.
\end{theorem}

As before, we derive Theorem \ref{thm:online} from a stronger statement that admits an inductive proof.
For a family $\cgI$ of intervals and a point $x\in\RR$, let $\cgI(x)$ denote the subfamily of $\cgI$ consisting of the intervals that have non-empty intersection with $[x,\infty)$.

\begin{lemma}\label{lem:online-induction}
Let\/ $k\geq 1$ and\/ $R$ be an interval of positive length.
In the on-line coloring game on triangle-free overlap graphs, Presenter has a strategy\/ $\Sigma(k,R)$ to construct a family of intervals\/ $\cgI$ together with a point\/ $x\in R$ that satisfy the following conditions:
\begin{enumerate}
\item every interval in\/ $\cgI$ is contained in the interior of\/ $R$,
\item the family\/ $\cgI(x)$ is a nested chain of intervals,
\item Painter is forced to use at least\/ $k$ colors on the intervals in\/ $\cgI(x)$.
\end{enumerate}
\end{lemma}

\begin{proof}
In the strategy $\Sigma(1,R)$, Presenter just puts a single interval $I$ in the interior of $R$.
Clearly, the family $\cgI=\{I\}$ together with an arbitrary point $x\in I$ satisfy the conditions of the lemma.

Now, let $k\geq 2$, and assume by induction that the strategy $\Sigma(k-1,R)$ exists and satisfies the conditions of the lemma for every $R$.
Fix an interval $R$ of positive length.
We describe the strategy $\Sigma(k,R)$.

\begin{figure}[t]
\begin{center}
\includegraphics[scale=1.1]{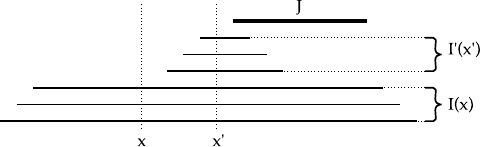}
\end{center}
\caption{The main idea of Presenter's strategy}
\label{fig:game-diagonal}
\end{figure}

First, Presenter plays according to the strategy $\Sigma(k-1,R)$ constructing a family $\cgI$ together with a point $x\in R$ such that $\cgI(x)$ is nested and Painter is forced to use at least $k-1$ colors on the intervals in $\cgI(x)$.
Let $R'=\bigcap\cgI(x)\cap[x,\infty)$.
Presenter continues with the strategy $\Sigma(k-1,R')$.
Thus another family $\cgI'$ is presented for which there is a coordinate $x'\in R'$ such that $\cgI'(x')$ is nested and Painter is forced to use at least $k-1$ colors on the intervals in $\cgI'(x')$.
If the sets of colors used by Painter on $\cgI(x)$ and $\cgI'(x')$ differ, then Painter has used at least $k$ colors on the family $\cgI(x)\cup\cgI'(x')$, which is equal to $(\cgI\cup\cgI')(x')$.
So suppose otherwise, that Painter has used the same sets of colors on $\cgI(x)$ and $\cgI'(x')$.
Presenter puts an interval $J$ contained in the interior of $\bigcap\cgI(x)\cap[x',\infty)$ and overlapping all intervals in $\cgI'(x')$ (see Figure \ref{fig:game-diagonal}).
The overlap graph remains triangle-free, as the family of neighbors of $J$, which is $\cgI'(x')$, is an independent set.
The family $\cgI(x)\cup\{J\}$ is nested, and we have $\cgI(x)\cup\{J\}=(\cgI\cup\cgI'\cup\{J\})(y)$ for any point $y$ in the interior of $J$ and to the right of all intervals in $\cgI'(x')$.
Painter has to color $J$ with a color distinct from those of the intervals in $\cgI'(x')$ and thus is forced to use at least $k$ colors on $\cgI(x)\cup\{J\}$.
The presented family $\cgI\cup\cgI'\cup\{J\}$ together with the point $y$ witness that the goal of the strategy $\Sigma(k,R)$ has been achieved.
\end{proof}

To obtain the smallest (shortest) strategy $\tilde\Sigma_k$ satisfying Theorem \ref{thm:online} that we know, we follow the strategy $\Sigma(k,R)$ constructed above and finally put one additional interval inside $[x,\infty)$ that overlaps all intervals in $\cgI(x)$.

Now we describe how to `encode' a strategy $\Sigma$ claimed by Theorem \ref{thm:online} into a triangle-free family $\cgF$ of axis-aligned rectangular frames with chromatic number greater than $k$.
This gives an alternative proof of the special case of Theorem \ref{thm:independent-scaling} for rectangular frames.
The idea is shown in Figure \ref{fig:game}.

For every interval $I$ occurring in any branch of the strategy $\Sigma$, define a rectangular frame $F(I)$ and insert it into $\cgF$ as follows.
For the very first interval $I_0$ presented by $\Sigma$, choose any $c,d\in\RR$ so that $c<d$, and define $F(I_0)$ to be the boundary of $I_0\times[c,d]$.
Now, let $I$ be an interval such that $F(I)$ is already defined as the boundary of $I\times[c_I,d_I]$.
Let $I_1,\ldots,I_r$ be the possible choices of the next interval put by $\Sigma$ after $I$, depending on the colors that Painter has chosen for $I$ and all preceding intervals in this branch of the strategy tree.
Choose $c_1,\ldots,c_r$ and $d_1,\ldots,d_r$ so that $c_I<c_1<d_1<\cdots<c_r<d_r<d_I$, and define $F(I_i)$ to be the boundary of $I_i\times[c_i,d_i]$ for $1\leq i\leq r$.
Repeat this procedure until every interval from every branch of $\Sigma$ has its rectangular frame defined.

\begin{figure}[t]
\begin{center}
\includegraphics[scale=1.1]{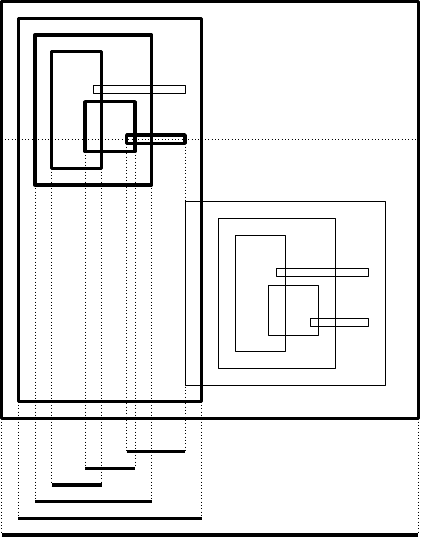}
\end{center}
\caption{A family of rectangular frames with $\chi=3$ and a branch of Presenter's strategy corresponding to the set of rectangles pierced by a horizontal line}
\label{fig:game}
\end{figure}

It is not difficult to see that any two rectangular frames $F(I),F(J)\in\cgF$ intersect if and only if the intervals $I$ and $J$ overlap and belong to a common branch in $\Sigma$.
In particular, any clique in $\cgF$ corresponds to a clique in the overlap graph presented in some branch of $\Sigma$.
Since $\Sigma$ always produces a triangle-free graph, the family $\cgF$ is triangle-free as well.
If $\chi(\cgF)\leq k$, then Painter could fix some $k$-coloring of $\cgF$ and color incoming intervals with the colors of the corresponding frames.
This way Painter would use at most $k$ colors of the intervals presented by $\Sigma$, which is a contradiction.
Hence $\chi(\cgF)>k$.

It should be noted that the proof of Lemma \ref{lem:other-induction} (for $X$ a rectangular frame) and the proof of Lemma \ref{lem:online-induction} together with the `encoding' described above provide equivalent descriptions of the same construction of triangle-free families of rectangular frames with large chromatic number.

\section{Problems}\label{sec:problems}

We already posed several problems concerning the chromatic number and independence number of triangle-free segment intersection graphs in \cite{PKK+14}.
Here we focus on classification of shapes with respect to whether triangle-free families of copies of the shape under particular transformations have bounded or unbounded chromatic number.

\begin{problem}
Provide a full characterization of compact arc-connected subsets $X$ of $\RR^2$ such that triangle-free families of homothetic copies of $X$ in the plane have bounded chromatic number.
\end{problem}

\begin{problem}
Is there a compact arc-connected subset $X$ of $\RR^2$ such that triangle-free families of translated copies of $X$ in the plane have unbounded chromatic number?
\end{problem}

\bibliographystyle{plain}
\bibliography{segments}

\end{document}